\documentclass[12pt]{amsart}
\usepackage{amsmath, amssymb, amsthm}

\title[The cyclic subgroup separability]{The~cyclic subgroup separability of~certain generalized free products of~two groups}

\author{P.~A.~Bobrovskii, E.~V.~Sokolov}

\newtheoremstyle{theorem}{10pt}{10pt}{\it}{\parindent}{\bf}{. }{ }{}
\theoremstyle{theorem}
\newtheoremstyle{corollary}{10pt}{10pt}{\it}{\parindent}{\bf}{. }{ }{}
\theoremstyle{corollary}
\newtheoremstyle{proposition}{10pt}{10pt}{\it}{\parindent}{\bf}{. }{ }{}
\theoremstyle{proposition}
\newtheoremstyle{lemma}{10pt}{10pt}{\it}{\parindent}{\bf}{. }{ }{}
\theoremstyle{lemma}
\newtheorem*{theorem}{Theorem}
\newtheorem*{corollary}{Corollary}
\newtheorem{proposition}{Proposition}

\begin{document}

\begin{abstract}
\hspace{-.3em}Free products of~two residually finite groups with~amalgamated retracts are considered. It is proved that a~cyclic subgroup of~such a~group is not finitely separable if, and only if, it is conjugated with~a~subgroup of~a~free factor which is not finitely separable in this factor. A~similar result is obtained for~the~case of~separability in~the~class of~finite \hbox{$p$-}groups.
\end{abstract}

\maketitle

\section{Statement of~results}
\footnotetext{\textit{Key words and~phrases:} cyclic subgroup separability, residual finiteness, residual \hbox{$p$-}finiteness, generalized free product, split extension.}
\footnotetext{2010 \textit{Mathematics Subject Classification:} primary 20E26, 20E06; secondary 20E22.}

The~object of~this paper is to study the~cyclic subgroup separability of~
free products of~two groups with~amalgamated subgroups, which are retracts of~free factors. We recall that a~subgroup~$H$ of~a~group~$G$ is a~\textit{retract} of~this group if~there exists a~subgroup~$F$, normal in~$G$, such that $G=HF$ and~$H\cap F=1$. In~other words, a~subgroup~$H$ is a~retract of~a~group~$G$ if~this group is the~splitting extension of~a~group~$F$ by~$H$.

Let us recall now that a~subgroup~$H$ is said to be \textit{finitely separable} in~a~group~$G$ if~for~any element $g\in G\backslash H$ there exists a~homomorphism~$\varphi$ of~$G$ onto a~finite group such that $g \varphi \notin H \varphi$. A~group~$G$ is called \textit{residually finite} if~its trivial subgroup is finitely separable. Fixing a~prime number~$p$ and~considering in~the~definitions above only homomorphisms onto finite \hbox{$p$-}groups instead of~homomorphisms onto arbitrary finite groups we obtain the~notions of~the~\hbox{$p$-}separability and~the~residual \hbox{$p$-}finiteness.

At last, we shall call a~subgroup~$H$ $p^\prime $-\textit{isolated} in~a~group~$G$ if~for~each element $g\in G$ and~for~each prime number $q \ne p$ $g \in H$ whenever $g^{q} \in H$. It is easy to see that the~property `to be $p^\prime $-isolated' is necessary for~the~\hbox{$p$-}separability. Therefore, we may consider only $p^\prime $-isolated subgroups while studying \hbox{$p$-}separ\-able cyclic subgroups.

If~$F$ is a~group, then we shall denote by~$\Delta (F)$ the~family of~all cyclic subgroups of~$F$ which are not finitely separable in~$F$ and~by~$\Delta_{p}(F)$ the~family of~all $p^\prime $-isolated cyclic subgroups of~$F$ which are not \hbox{$p$-}separ\-able in~this group. The~main result of~this paper is the~following

\begin{theorem}
Let
$$
G=\langle A * B;\ H = K,\ \varphi \rangle
$$
be the~free product of~groups~$A$ and~$B$ with~subgroups~$H$ and~$K$ amalgamated according to 
an isomorphism~$\varphi$. Let also $H$ be a~retract of~$A$ and~$K$ be a~retract of~$B$.

\textup{I.} If~$A$ and~$B$ are residually finite, then a~cyclic subgroup of~$G$ is finitely separable in~this group if, and~only if, it is not conjugated with~any subgroup of~the~family $\Delta (A) \cup \Delta (B)$.

\textup{II.} If~$A$ and~$B$ are residually \hbox{$p$-}finite, then a~$p^\prime $-isolated cyclic subgroup of~$G$ is \hbox{$p$-}separ\-able in~this group if, and~only if, it is not conjugated with~any subgroup of~the~family $\Delta _{p}(A) \cup \Delta _{p}(B)$.
\end{theorem}

The~first part of~this theorem generalizes the~result of~R.~B.~J.~T.~Allenby and~R.~J.~Gregorac~\cite{li01} who showed that a~free product with~amalgamated retracts of~two $\pi_{c}$-groups (i.~e., groups with~all cyclic subgroups being finitely separable) is a~$\pi _{c}$-group. A~similar statement for~an~arbitrary number of~free factors was proved by~G.~Kim in~\cite{li04}.

We note that for~any generalized free product 
$$
G=\langle A * B;\ H=K,\ \varphi \rangle,
$$
if a~cyclic subgroup is conjugated with~a~subgroup of~$\Delta (A) \cup \Delta (B)$ ($\Delta _{p}(A) \cup \Delta _{p}(B)$), it is certainly not finitely separable (respectively \hbox{$p$-}separ\-able) in~$G$. Thus, the~theorem formulated states in~fact the~maximality of~the~families of~finitely separable and~\hbox{$p$-}separ\-able cyclic subgroups of~$G$.

\begin{corollary}
The~group~$G$ satisfying the~condition of~the~main theorem 
is residually finite (residually \hbox{$p$-}finite) if, and~only if, the~groups~$A$ and~$B$ are residually finite (respectively, residually \hbox{$p$-}finite).
\end{corollary}

\begin{proof}
The residual finiteness of~$A$ and~$B$, being equivalent to the~
finite separability of~the~trivial subgroup~$E$ in~this groups, means this subgroup does not belong to the~family $\Delta (A) \cup \Delta (B)$ and, since it is invariant in~$G$, is not conjugated with~any subgroup of~this family. Therefore, by~the~statement~I of~the~main theorem $E$ is finitely separable in~$G$, i.~e. $G$ is residually finite.

Similarly, the~\hbox{$p$-}separability of~$A$ and~$B$ implies $E$ does not belong to the~family $\Delta _{p}(A) \cup \Delta _{p}(B)$. In~addition, this subgroup is $p^\prime $-isolated in~the~free factors, what means they have no elements of~prime orders not equal to $p$. By the~torsion theorem for~generalized free products (see, e.~g., \cite[section~IV, theorem~1.6]{li05}) $G$ does not contain such elements too. Hence, $E$ is $p^\prime $-isolated in~$G$ and~is \hbox{$p$-}separ\-able in~this group according to the~statement~II of~the~main theorem.

Thus, the~sufficiency is proved while the~necessity is clear.
\end{proof}

We note that the~criterion of~the~residual finiteness formulated in~this corollary was proved by~J.~Boler and~B.~Evans in~\cite{li03}.

\section{Proof of~the~main theorem}

Let
$$
G=\langle A * B;\ H=K,\ \varphi \rangle
$$
be the~free product of~groups $A$ and~$B$ with~subgroups $H$ and~$K$ amalgamated according to an~isomorphism~$\varphi$. Following to~\cite{li02} and~\cite{li06} we shall call normal subgroups of~finite index $R \leqslant A$ and~$S \leqslant B$:

a) \textit{$(H, K, \varphi )$-compatible} if~$(R \cap H)\varphi = S \cap K$;

b) \textit{$(H, K, \varphi , p)$-compatible}, where $p$ is a~fixed prime number, if~there exist sequences of~subgroups
$$
R=R_{{\rm 0}} \leqslant \ldots \leqslant R_{m} = A,\ 
S=S_{{\rm 0}} \leqslant \ldots \leqslant S_{n}=B
$$
such that:

1) $R_{i}$, $S_{j}$ are normal subgroups of~$A$ and~$B$ respectively 
($0 \leqslant i \leqslant m$, $0 \leqslant j \leqslant n$);

2) $\vert R_{i+1}/R_{i}\vert = \vert S_{j+1}/S_{j}\vert = p$ 
($0 \leqslant i \leqslant m-1$, $0 \leqslant j \leqslant n-1$);

3) $\varphi $ maps the~set
$$
\{R_{i} \cap H \mid 0 \leqslant i \leqslant m\}
$$
onto the~set 
$$
\{S_{j} \cap K \mid 0 \leqslant j \leqslant n\}.
$$

Let $\Omega $ be the~family of~all pairs $(H, K, \varphi )$-compatible 
subgroups and~$\Omega _{p}$ be the~family of~all pairs $(H, K, \varphi , p)$-compatible subgroups. We shall denote by~$\Omega (A)$, $\Omega _{p}(A)$, $\Omega (B)$, $\Omega _{p}(B)$ the~projections of~these families onto~$A$ and~$B$.

Now we want to prove that if~$H$ and~$K$ are retracts of~the~free factors, then $\Omega (A)$ and~$\Omega (B)$ coincide with~the~families $\Theta (A)$ and~$\Theta (B)$ of~all normal subgroups of~finite index of~$A$ and~$B$ respectively and~$\Omega _{p}(A)$ and~$\Omega _{p}(B)$ coincide with~the~families $\Theta _{p}(A)$ and~$\Theta _{p}(B)$ of~all normal subgroups of~finite \hbox{$p$-}index of~$A$ and~$B$. For~this purpose we need the~following

\begin{proposition}\label{prop1}
Let $Y$ be a~retract of~a~group~$X$ and~$F$ be a~normal 
subgroup of~$X$ such that $X=YF$ and~$Y \cap F=1$. Let also $N$ be a~normal subgroup of~$Y$. Then $NF$ is a~normal subgroup of~$X$, \hbox{$NF \cap Y = N$} and~$X/NF \cong Y/N$. In~particular, $[X:NF]=[Y:N]$.
\end{proposition}

\begin{proof}
Since $F$ is normal in~$X$ and~$N$ is normal in~$Y$,
$$
(NF)^{X} \subseteq N^{X}F=(N^{Y})^{F}F \subseteq N^{F}F \subseteq NF.
$$

The~triviality of~the~intersection of~$Y$ and~$F$ implies the~following:
$$
X/NF=YF/NF \cong Y/N(Y \cap F)=Y/N.
$$

At~last, considering an~element $x \in NF \cap Y$ and~writing it in~the~
form $x=yf$, where $y \in N$, $f \in F$, we get $f=y^{-1}x \in Y$. 
But~$Y \cap F=1$, so $f=1$ and~$x=y \in N$. Thus, $NF \cap Y \subseteq N$ 
and, since the~inverse inclusion is clear, $NF \cap Y=N$, as~claimed.
\end{proof}

Returning to the~proof of~the~main theorem we consider an~arbitrary normal subgroup $R$ of~finite index of~$A$. The~subgroup $P=R \cap H$ is normal in~$H$, so $Q=P\varphi $ is a~normal subgroup of~finite index of~$K$.

Since $K$ is a~retract of~$B$, there exists a~normal subgroup $F \leqslant B$ satisfying the~conditions $B=KF$ and~$K \cap F=1$. By~Proposition~\ref{prop1} $S=QF$ is a~normal subgroup of~finite index of~$B$ and
$$
S \cap K=Q=(R \cap H)\varphi.
$$
Thus, $(R, S) \in \Omega $, and~so $R \in \Omega (A)$.

Suppose now that $R$ has \hbox{$p$-}index in~$A$.

It is well known that every finite \hbox{$p$-}group possesses a~normal series with~the factors of~order $p$. Let
$$
1=R_{0}/R \leqslant \ldots \leqslant R_{m}/R=A/R
$$
be such a~series of~the~factor-group~$A/R$. Then 
$$
R=R_{0} \leqslant \ldots \leqslant R_{m}=A
$$
is a~sequence of~normal subgroups of~$A$ and, since
$$
R_{i+1}/R_{i} \cong (R_{i+1}/R)/(R_{i}/R),
$$
its factors have the~order $p$ too.

Let us put $P_{i}=R_{i} \cap H$, $Q_{i}=P_{i}\varphi $ and~$S_{i}=Q_{i}F$, where \hbox{$0 \leqslant i \leqslant m$} and~$F$ is the~subgroup of~$B$ defined earlier. Then $P_{i}$ and~$Q_{i}$ are normal and~have finite index in~$H$ and~$K$ respectively. Therefore, by~Proposition~\ref{prop1} $S_{i}$ are normal and~have finite index in~$B$. Moreover,
$$
S_{i+1}/S_{i}=Q_{i+1}F/Q_{i}F \cong Q_{i+1}/Q_{i}(Q_{i+1} \cap F).
$$
But~$Q_{i+1} \leqslant K$ and~$K \cap F=1$, so 
$S_{i+1}/S_{i} \cong Q_{i+1}/Q_{i}$ and~$\vert S_{i+1}/S_{i}\vert =\vert Q_{i+1}/Q_{i}\vert $.

We note, further, that
\begin{align*}
Q_{i+1}/Q_{i} &\cong P_{i+1}/P_{i} \\
&= (R_{i+1} \cap H)/(R_{i} \cap H) \\
&= (R_{i+1} \cap H)/(R_{i} \cap H)(R_{i+1} \cap H \cap R_{i}) \\
&\cong (R_{i+1} \cap H)R_{i}/(R_{i} \cap H)R_{i} \\
&= (R_{i+1} \cap H)R_{i}/R_{i} \\
&\leqslant R_{i+1}/R_{i}.
\end{align*}

So the~order of~the~factor-group $Q_{i+1}/Q_{i}$ and, hence, the~order of~the~factor-group $S_{i+1}/S_{i}$ divides the~order of~the~factor-group\linebreak $R_{i+1}/R_{i}$, which is equal to~$p$. Since $p$ is a~prime number, it follows that either $S_{i+1}=S_{i}$ or~$\vert S_{i+1}/S_{i}\vert =p$. If we remove from the~sequence 
$$
S=S_{0} \leqslant \ldots \leqslant S_{n}=B
$$ 
repeating members, then the~second equality will always take place.

At~last, by~a~construction 
$$
S_{i} \cap K=Q_{i}=(R_{i} \cap H)\varphi.
$$ 
Hence, \hbox{$R \in \Omega _{p}(A)$}.

Thus, all normal subgroups of~finite index of~$A$ are contained in~$\Omega (A)$ and~all normal subgroups of~finite \hbox{$p$-}index are contained in~$\Omega _{p}(A)$. Since the~inverse inclusions are clear, we get the~required equalities $\Theta (A)=\Omega (A)$ and~$\Theta _{p}(A)=\Omega _{p}(A)$. The~arguments for~the~group $B$ are just the~same.

\medskip

We shall say further that a~subgroup $Y$ of~a~group $X$ is 
\textit{separable by~a~family} $\Psi $ \textit{of~normal subgroups of}~$X$ if 
$$
\bigcap_{N \in \Psi} YN = Y.
$$

Recall (see \cite{li02}) that $\Psi $ is a~$Y$-\textit{filtration} if~the~subgroups $Y$ and~$\{1\}$ are separable by~it.

Let us denote by~$\Lambda (A)$ and~$\Lambda (B)$ the~families of~all cyclic subgroups of~$A$ and~$B$ which are not separable by~$\Omega (A)$ and~$\Omega (B)$ respectively, and~by~$\Lambda _{p}(A)$ and~$\Lambda _{p}(B)$ the~families of~all $p^\prime $-isolated cyclic subgroups of~$A$ and~$B$ which are not separable by~$\Omega _{p}(A)$ and~$\Omega _{p}(B)$. The~following general statements take place.

\begin{proposition}\cite[theorem~1.2]{li07}\label{prop2}
Let the~family $\Omega (A)$ be an~$H$-fil\-tra\-tion and~the~family $\Omega (B)$ be a~$K$-fil\-tra\-tion. Then a~cyclic subgroup of~$G$ is finitely separable if~it is not conjugate with~any subgroup of~the~family $\Lambda (A) \cup \Lambda (B)$. \qed
\end{proposition}

\begin{proposition}\cite[theorem~1.6]{li07}\label{prop3}
Let the~family $\Omega _{p}(A)$ be an~$H$-fil\-tra\-tion and~the~family $\Omega _{p}(B)$ be a~$K$-fil\-tra\-tion. Then a~$p^\prime $-isolated cyclic subgroup of~$G$ is \hbox{$p$-}separ\-able if~it is not conjugate with~any subgroup of~the~family $\Lambda _{p}(A) \cup \Lambda _{p}(B)$. \qed
\end{proposition}

As~can be easily shown the~finite separability\,(\hbox{$p$-}separability)\,of~a~subgroup in~a~group is equivalent to its separability by~the~family of~all normal subgroups of~finite index (respectively finite \hbox{$p$-}index) of~this group. Since, in~our case, $\Omega (A)$ and~$\Omega (B)$ exactly coincide with~the~families~$\Theta (A)$ and~$\Theta (B)$ of~all normal subgroups of~finite index of~$A$ and~$B$, the~condition of~Proposition~\ref{prop2} turns out to be equivalent to simultaneous fulfillment of~the~following two statements:

\smallskip

$1$) $A$ and~$B$ are residually finite;

$2$) $H$ and~$K$ are finitely separable in~the~free factors.

\smallskip

In~just the~same way the~condition of~Proposition~\ref{prop3} is equivalent to 
simultaneous fulfillment of~the~following two statements:

\smallskip

$1_{p}$) $A$ and~$B$ are residually \hbox{$p$-}finite;

$2_{p}$) $H$ and~$K$ are \hbox{$p$-}separ\-able in~the~free factors.

\smallskip

In~addition, the~equalities
$$
\Theta (A)=\Omega (A),\ \Theta (B)=\Omega (B),\ 
\Theta _{p}(A)=\Omega _{p}(A),\ \Theta _{p}(B)=\Omega _{p}(B)
$$
imply
$$
\Lambda (A)=\Delta (A),\ \Lambda (B)=\Delta (B),\
\Lambda _{p}(A)=\Delta _{p}(A),\ \Lambda _{p}(B)=\Delta _{p}(B).
$$

So we can reduce both statements of~the~main theorem to~Propositions~\ref{prop2} and~\ref{prop3} if~only prove that the~conditions~$2$) and~$2_{p}$) follow from the~conditions~$1$) and~$1_{p}$) respectively.

\begin{proposition}\label{prop4}
A~retract of~a~residually finite group is finitely separable in~this group. A~retract of~a~residually \hbox{$p$-}finite group is \hbox{$p$-}separ\-able in~this group.
\end{proposition}

\begin{proof}
We shall prove both statements simultaneously.

Let $Y$ be a~retract of~a~residually finite (residually \hbox{$p$-}finite) group $X$ and~$F$ be a~normal subgroup of~$X$ such that $X=YF$ and~$Y \cap F=1$. Let also $x \in X\backslash Y$ be an~arbitrary element. To prove $Y$ is finitely separable (\hbox{$p$-}separ\-able) we only need to find a~normal subgroup~$L$ of~$X$ having finite index (finite \hbox{$p$-}index) in~this group and~such that $x \notin YL$.

Write $x$ in~the~form $x=yf$, where $y \in Y$, $f \in F$. Since $x \notin Y$, $f \ne 1$. So, using the~residual finiteness (the residual \hbox{$p$-}finiteness) of~$X$, one can find a~normal subgroup~$N$ of~finite index 
(respectively finite \hbox{$p$-}index) of~this group such that $f \notin N$. Let us put $U=N \cap F$, $V=N \cap Y$, $L=VU$ and~$M=VF$.

By~Proposition~\ref{prop1} $M$ is normal in~$X$ and, since
$$
G/M \cong Y/V=Y/N \cap Y \cong YN/N \leqslant G/N,
$$
it has finite index (finite \hbox{$p$-}index) in~this group. $L$ possesses the~same properties: it follows from the~easily verifying equality
$$
L=VU=V(F \cap N)=VF \cap N=M \cap N
$$
and~from the~inclusion
$$
M/L=M/M \cap N \cong MN/N \leqslant G/N.
$$

Suppose now that $x \in YL$.

Since $YL=YVU=YU$, one can write $x$ in~the~form $x=zu$, where $z \in Y$, $u \in U$. Then $yf=zu$ and~$z^{-1}y=uf^{-1}$. But~$z^{-1}y \in Y$, $uf^{-1} \in F$ and~$Y \cap F=1$. Therefore, $f=u \in N$, what contradicts the~choice of~$N$.

Thus, $x \notin YL$ and~$L$ is a~required subgroup. This ends the~proof of~Proposition and~the~main theorem.
\end{proof}

\end{document}